\documentclass[letterpaper,10pt]{amsart}

\usepackage[all]{xy}                        %

\CompileMatrices                            

\UseTips                                    

\input xypic
\usepackage[bookmarks=true]{hyperref}       

\usepackage{amssymb,latexsym,amsmath,amscd}
\usepackage{xspace}

\usepackage{graphicx}
\usepackage{color}

\reversemarginpar

\theoremstyle{plain}
\newtheorem{theorem}{Theorem}[section]
\newtheorem*{theorem*}{Theorem}
\newtheorem{proposition}[theorem]{Proposition}
\newtheorem{corollary}[theorem]{Corollary}
\newtheorem{lemma}[theorem]{Lemma}

\theoremstyle{definition}
\newtheorem{definition}[theorem]{Definition}

\newtheorem{notation}[theorem]{Notation}

\newtheorem{remark}[theorem]{Remark}

\newcommand{\enm}[1]{\ensuremath{#1}}          %
\newcommand{\op}[1]{\operatorname{#1}}
\newcommand{\cal}[1]{\mathcal{#1}}

\newcommand{\CC}{\enm{\mathbb{C}}}
\newcommand{\II}{\enm{\mathbb{I}}}
\newcommand{\NN}{\enm{\mathbb{N}}}

\newcommand{\QQ}{\enm{\mathbb{Q}}}
\newcommand{\ZZ}{\enm{\mathbb{Z}}}
\newcommand{\EE}{\enm{\mathbb{E}}}

\newcommand{\PP}{\enm{\mathbb{P}}}

\newcommand{\Aa}{\enm{\cal{A}}}

\newcommand{\Dd}{\enm{\cal{D}}}
\newcommand{\Ee}{\enm{\cal{E}}}
\newcommand{\Ff}{\enm{\cal{F}}}
\newcommand{\Gg}{\enm{\cal{G}}}

\newcommand{\Ii}{\enm{\cal{I}}}
\newcommand{\Jj}{\enm{\cal{J}}}

\newcommand{\Ll}{\enm{\cal{L}}}

\newcommand{\Nn}{\enm{\cal{N}}}
\newcommand{\Oo}{\enm{\cal{O}}}

\newcommand{\Rr}{\enm{\cal{R}}}
\newcommand{\Ss}{\enm{\cal{S}}}
\newcommand{\Tt}{\enm{\cal{T}}}

\renewcommand{\phi}{\varphi}
\renewcommand{\theta}{\vartheta}
\renewcommand{\epsilon}{\varepsilon}

\newcommand{\Pic}{\op{Pic}}

\newcommand{\Ext}{\op{Ext}}

\newcommand{\Aut}{\op{Aut}}
\newcommand{\Supp}{\op{Supp}}

\renewcommand{\to}[1][]{\xrightarrow{\ #1\ }}

\newcommand{\old}[1]{}

\begin{document}

\title[Curves on Segre threefolds]{Curves on Segre threefolds}

\author{Edoardo Ballico, Kiryong Chung and Sukmoon Huh}

\address{Universit\`a di Trento, 38123 Povo (TN), Italy}
\email{edoardo.ballico@unitn.it}

\address{Department of Mathematics Education, Kyungpook National University, 80 Daehakro, Bukgu, Daegu 41566, Korea}
\email{krchung@knu.ac.kr}

\address{Sungkyunkwan University, Suwon 440-746, Korea}
\email{sukmoonh@skku.edu}


\thanks{The first author is partially supported by MIUR and GNSAGA of INDAM (Italy). The second author is supported by Kyungpook National University Research Fund, 2015. The third author is supported by Basic Science Research Program 2015-037157 through NRF funded by MEST}


\begin{abstract}
We study locally Cohen-Macaulay curves of low degree in the Segre threefold with Picard number three and investigate the irreducible and connected components respectively of the Hilbert scheme of them. We also discuss the irreducibility of some moduli spaces of purely one-dimensional stable sheaves and apply the similar argument to the Segre threefold with Picard number two.
\end{abstract}

\maketitle
\tableofcontents
\section{Introduction}
In this paper we study curves in Segre threefolds over the field of complex numbers $\CC$. There are three types of Segre threefolds: $\PP^3$, $\PP^2 \times \PP^1$ and $\PP^1 \times \PP^1 \times \PP^1$. In $\PP^3$ the structure of the Hilbert scheme of curves has been densely studied by many authors in the last half century. In \cite{Ha1} the connectedness of the Hilbert scheme of curves is proven for the fixed degree and genus of curves, although it is classically known that the locus of smooth curves may not be connected. Recently there have been increased interests on the connectedness of the Hilbert scheme of locally Cohen-Macaulay curves. Up to now, the connectedness has been established only for very small degree \cite{NS} or for very large genus \cite{HMP}. We recommend to see \cite{hls} for further results and the state of the art on this problem. 

Our main concern is on the connectedness of Hilbert schemes of locally Cohen-Macaulay curves in $X=\PP^1 \times \PP^1 \times \PP^1$ with very small degree. Smooth curves are often the first to be studied and by the Hartshorne-Serre correspondence globally generated vector bundles on $X$ can have very close relation with smooth curves in $X$. There is a classification of globally generated vector bundles on $X$ with low first Chern class accomplished by the classification of smooth curves in $X$ with very small degree \cite{BHM}. One of the advantages in the study of curves in $X$ is that some irreducible components that might appear in the Hilbert scheme in $\PP^3$ may disappear in $X$ so that we can get simpler description of Hilbert schemes; for example, the Hilbert scheme of curves in $\PP^3$ of degree $3$ and genus $0$ has two irreducible components, one with twisted cubics and the other with planar cubics plus extra point (see \cite{PS}). The latter case cannot occur in $X$ because $X$ is scheme-theoretically cut out by quadrics. 

Our main result is as follows:
\begin{theorem}
Let $\mathbf{H}(e_1, e_2, e_3, \chi)_{+, \mathrm{red}}$ be the reduced Hilbert scheme of locally Cohen-Macaulay curves $C$ in $X$ with tridegree $(e_1, e_2, e_3)$ and $\chi(\Oo_C)=\chi$. 
\begin{itemize}
\item[(i)] $\mathbf{H}(2,0,0,a)_{+, \mathrm{red}}$ is irreducible and rational for $a\ge 2$;
\item[(ii)] $\mathbf{H}(2,1,0,a)_{+, \mathrm{red}}$ has the two irreducible components for $a\ge 3$;
\item[(iii)] $\mathbf{H}(1,1,1,a)_{+, \mathrm{red}}$ is irreducible and rational for $a\in \{1,3\}$, while $\mathbf{H}(1,1,1,2)_{+, \mathrm{red}}$ has the three connected components that are rational;
\item[(iv)] $\mathbf{H}(2,1,1,1)_{+, \mathrm{red}}$ is irreducible and rational. 
\end{itemize}
\end{theorem}

The main ingredient in the study of Hilbert schemes of locally Cohen-Macaulay curves with low degree is a rational ribbon, i.e. a double structure on $\PP^1$. Rational ribbons and their canonical embeddings were studied in \cite{be} and we adopt their results to prove the irreducibility of the Hilbert schemes of double lines in $X$. Then we investigate the intersecting property of the double lines with other lines in $X$ to investigate irreducible and connected components of the Hilbert schemes respectively. We recommend to see \cite{M,NNS} for the study on the families of double lines in projective spaces. It should be noted that the irreducibility of the space of curves with fixed cohomology in $\PP^3$ is investigated in \cite{Bol}. And the description of the other type of Hilbert scheme is studied in \cite{AV}. 

Let us summarize here the structure of this paper. In section $2$, we introduce the definitions and main properties that will be used throughout the paper, mainly those related to Segre threefold, Hilbert tripolynomial and Hilbert schemes of locally Cohen-Macaulay curves. In section $3$, we pay attention to the Hilbert schemes of curves with tridegree $(2,0,0)$ and conclude their irreducibility using the double structure on $\PP^1$. We end the section with the description of the intersection of the double lines with other lines in $X$, which will be used later on. In section $4$, we move forward to the Hilbert schemes of curves with tridegree $(2,1,0)$, $(1,1,1)$ and $(2,1,1)$, and describe their irreducible and connected components respectively. In the proof of irreducibility of $\mathbf{H}(2,1,1,1)_{+, \mathrm{red}}$ we use the moduli of stable maps. In section $5$, we deal with the wall-crossings among moduli space of stable pairs \cite{He} to investigate the irreducibility of some moduli spaces of stable purely $1$-dimensional sheaves on $X$. It turns out that there are no wall-crossings in our examples and it enables us to reach our conclusion in an easy way based on the results on Hilbert scheme of curves. Finally in section $6$, we apply our arguments to the case of Segre threefold with Picard number two.

We would like to thank the anonymous referee for pointing out a number of critical mistakes in the first version of this article.


\section{Preliminaries}
For three $2$-dimensional vector spaces $V_1, V_2, V_3$ over the field of complex numbers $\CC$, let $X\cong \PP (V_1) \times \PP (V_2) \times \PP (V_3)$ and then it is embedded into $\PP^7\cong \PP(V_0)$ by the Segre map where $V_0=V_1 \otimes V_2 \otimes V_3$. It is known that $X$ is the only Del Pezzo with the maximal Picard number $\varrho (X)=3$. The intersection ring $A(X)$ is isomorphic to $A(\PP^1) \otimes A(\PP^1) \otimes A(\PP^1)$ and so we have
$$A(X) \cong \ZZ[t_1, t_2, t_3]/(t_1^2, t_2^2, t_3^2).$$
We may identify $A^1(X)\cong \ZZ^{\oplus 3}$ by $a_1t_1+a_2t_2+a_3t_3 \mapsto (a_1, a_2, a_3)$. Similarly we have $A^2(X) \cong \ZZ^{\oplus 3}$ by $e_1t_2t_3+e_2t_3t_1+e_3t_1t_2\mapsto (e_1, e_2, e_3)$ and $A^3(X) \cong \ZZ$ by $ct_1t_2t_3 \mapsto c$.

Let us denote the natural projection of $X$ to $i^{\mathrm{th}}$ factor by $\pi_i: X \rightarrow \PP^1$ and we denote $\pi_1^*\Oo_{\PP^1}(a_1)\otimes \pi_2^*\Oo_{\PP^1}(a_2)\otimes \pi_3^* \Oo_{\PP^1}(a_3)$ by $\Oo_X(a_1, a_2, a_3)$. Then $X$ is embedded into $\PP^7$ by the complete linear system $|\Oo_X(1,1,1)|$ as a subvariety of degree $6$ since $(1,1,1)^3=6$. We also denote $\Ee \otimes \Oo_X(a_1, a_2, a_3)$ by $\Ee(a_1, a_2, a_3)$ for a coherent sheaf $\Ee$ on $X$. We also let $\pi_{ij}: X \rightarrow \PP^1\times \PP^1$ denote the projection to $(i,j)$-factor, i.e. $\pi _{ij}(o_1,o_2,o_3) =(o_i,o_j)$ for $(o_1,o_2,o_3)\in X$. 

\begin{proposition}\label{trip}
For a $1$-dimensional sheaf $\Ff$ on $X$, there exists a tripolynomial $\chi_{\Ff}(x,y,z)\in \QQ[x,y,z]$ of degree $1$ such that 
 $$\chi (\Ff(u,v,w))=\chi_{\Ff}(u,v,w)$$
 for all $(u,v,w)\in \ZZ^{\oplus 3}$. 
\end{proposition}
\begin{proof}
This follows verbatim from the proof of \cite[Proposition 2]{BH}. Let $at+b\in \QQ[t]$ be the Hilbert polynomial of $\Ff$ with respect to $\Oo_X(1,1,1)$. Take a divisor $D_1\in |\Oo_X(1,0,0)|$ such that $D_1$ misses any $0$-dimensional components (embedded or isolated) of $\mathrm{Supp}(\Ff)$ and does not contain any component of the $1$-dimensional reduced scheme associated to $\mathrm{Supp}(\Ff)$. Then it gives us an injective map $j_{D_1}: \Ff(t,t,t) \rightarrow \Ff(t+1, t,t)$ and so we have an exact sequence
\begin{equation}\label{eqa1}
0\to \Ff (t,t,t) \to \Ff(t+1, t,t) \to \Ff (t+1, t,t)\otimes \Oo_{D_1} \to 0.
\end{equation}
Similarly let us fix other divisors $D_2\in |\Oo_X(0,1,0)|,D_3\in |\Oo_X(0,0,1)|$ and $D\in |\Oo_X(1,1,1)|$ to define maps $j_{D_2},j_{D_3}$ and $j_D$ with exact sequences as in (\ref{eqa1}). 

Let us set $l:=h^0(\Ff(t+1, t,t)\otimes \Oo_{D_1})$, $m:=h^0(\Ff(t, t+1,t)\otimes \Oo_{D_2})$ and $n:=h^0(\Ff(t, t,t+1)\otimes \Oo_{D_3})$ that are independent on $t$. We claim that $\chi (\Ff(u,v,w))=lu+mv+nw+b$ for all $(u,v,w)\in \ZZ^{\oplus 3}$. From the exact sequence for $D$, we have $l+m+n=h^0(\Ff(t+1, t+1, t+1)\otimes \Oo_D)=a$ and so the claim is true if $u=v=w$. In general let us assume $u\ge v \ge w$ without loss of generality. Then we get $\chi(\Ff(u,v,w))=\chi(\Ff(w,w,w))+l(u-w)+m(v-w)$, using the exact sequences for $D_1$ and $D_2$ several times. 
\end{proof}

\begin{definition}
We call the linear tripolynomial in Proposition \ref{trip} the {\it Hilbert tripolynomial} of $\Ff$ for a purely $1$-dimensional sheaf $\Ff$, i.e. $\chi_{\Ff}(x,y,z)=e_1x+e_2y+e_3z+\chi$ for some $(e_1, e_2, e_3,\chi)\in \ZZ^{\oplus 4}$. In particular, the Hilbert polynomial of $\Ff$ with respect to $\Oo_X(1,1,1)$ is defined to be $\chi_{\Ff}(t)=\chi_{\Ff}(t,t,t)$. We also call $\chi_{\Oo_C}(x,y,z)$ the Hilbert tripolynomial of a curve $C$.
\end{definition}

Let $\mathbf{H}{(e_1, e_2, e_3, \chi)}$ be the Hilbert scheme of curves in $X$ with the Hilbert tripolynomial $e_1x+e_2y+e_3z+\chi$, and let $\mathbf{H}{(e_1, e_2, e_3, \chi)}^{\mathrm{sm}}$ be the open locus corresponding to smooth and connected curves. 

\begin{definition}
A {\it locally Cohen-Macaulay (for short, locally CM)} curve in $X$ is a $1$-dimensional subscheme $C\subset X$ whose irreducible components are all $1$-dimensional and that has no embedded points. Equivalently $\Oo_C$ is purely $1$-dimensional. 
\end{definition}

We denote by $\mathbf{H}(e_1, e_2, e_3, \chi)_+$ the subset of $\mathbf{H}(e_1, e_2, e_3, \chi)$ parametrizing the locally CM curves with no isolated point. In particular we have $\mathbf{H}(e_1, e_2, e_3, \chi)^{\mathrm{sm}} \subset \mathbf{H}(e_1, e_2, e_3, \chi)_+$. 
 
\begin{remark}\label{tr1}
Let $C$ be an integral projective curve. By the universal property of fibered product there is a bijection between the morphisms $u: C\rightarrow X$ and the triples $(u_1,u_2,u_3)$ with $u_i: C \rightarrow \PP^1$ any morphism. The set $u(C)$ is contained in a $2$-dimensional factor of $X$ if and only if one of the $u_1,u_2,u_3$ is constant. We say that a constant map has degree zero. With this convention we may associate to any $u$ a triple $(\deg (u_1),\deg (u_2),\deg (u_3))\in \ZZ_{\ge 0}^{\oplus 3}$ and $u(C)$ is a curve if and only if $(\deg (u_1),\deg (u_2),\deg (u_3)) \ne (0,0,0)$. Now assume that $u$ is birational onto its image. With this assumption for all $(a_1,a_2,a_3)\in \ZZ^{\oplus 3}$, we have 
$$u(C)\cdot \Oo _X(a_1,a_2,a_3) =a_1\deg (u_1)+a_2\deg (u_2)+a_3\deg (u_3).$$
In particular the degree of the curve $u(C)$ is $\deg (u_1)+\deg (u_2)+\deg (u_3)$.
\end{remark}

\begin{lemma}\label{a1}
Let $C\subset X$ be a locally CM curve with the tridegree $(e_1, e_2, e_3)$. If the tridegree of $C_{\mathrm{red}}$ is $(b_1, b_2, b_3)$ with $b_i=0$ for some $i$, then we have $e_i=0$.
\end{lemma}
\begin{proof}
In general, if $u_i : C \rightarrow \PP^1$ is the $i^{\mathrm{th}}$-projection with $f_i$ the length of the generic fibre of $u_i$, then $C$ has tridegree $(f_1, f_2, f_3)$. Now let us assume $i=3$, i.e. $b_3=0$. The restriction of the projection $\pi_{3|C}: C\rightarrow \PP^1$ has degree $e_3$. Similarly $\pi_{3|C_{\mathrm{red}}}: C_{\mathrm{red}} \rightarrow \PP^1$ has degree $b_3=0$. Thus $\pi_{3|C_{\mathrm{red}}}$ has finite image and so does $\pi_{3|C}$. In particular, we have $e_3=0$.
\end{proof}


\section{Double lines}

\begin{notation}
Throughout this article by a line we mean  a CM curve with tridegree $(1,0,0)$, $(0,1,0)$ or $(0,0,1)$. A double line is by definition a double structure on a line. For each $a\in \ZZ$, let $\Dd _a$ be the subset of  $\mathbf{H}(2,0,0,a)_+$ parametrizing the double lines.
\end{notation}

For the moment we take $\Dd _a$ as a set. In each case it would be clear which scheme-structure is used on it. Since $X$ is a smooth 3-fold, \cite[Remark 1.3]{bf} says that each $B\in \Dd _a$ is obtained by the Ferrand construction and in particular it is a ribbon in the sense of \cite{be} with a line of tridegree $(1,0,0)$ as its support. Let $C_a$ be the unique split ribbon with $\chi (\Oo _{C_a}) = a$ and every ribbon is split for $a\ge 1$ by \cite[Theorem 1.2]{bf}. Each $f\in \mathrm{Aut}(C_a)$ induces an automorphism $\widetilde{f}$ of $\PP^1$ and the map $f\mapsto \widetilde{f}$ is surjective. Thus we get $\dim \mathrm{Aut}(C_a) \ge 3$. Since $C_a$ is equipped with a specification of a normal direction at each point of $\PP^1$, so we have $\mathrm{Aut}(C_a) \cong \mathrm{Aut}(\PP^1)$. In particular, we have $\dim \mathrm{Aut}(C_a)=3$.

\begin{theorem}\label{o1}
The description on $\Dd_a$ is as follows:
\begin{enumerate}
\item $\Dd_a$ is non-empty if and only if $a\ge 2$. It is parametrized by an irreducible and rational variety of dimension $2a-1$.
\item We have $\Dd _a = \mathbf{H}(2,0,0,a)_+$ for $a\ge 3$
\item $\mathbf{H}(2,0,0,2)_+$ is isomorphic to $\mathbf{Hilb}^2(\PP^1\times \PP^1)$, the Hilbert scheme of two points in $\PP^1 \times \PP^1$. In particular, it is smooth, irreducible, rational and of dimension $4$. 
\end{enumerate}
\end{theorem}

\begin{proof}
Fix $[A]\in \Dd _a$ and set $L:= A_{\mathrm{red}}$. The curve $A$ is a locally CM curve
of degree $2$ with $L \cong \PP^1$ as its support (see \cite[Remark 1.3]{bf}) and so it is a ribbon in the sense of \cite{be}. The projection $\pi_1$ induces a morphism $\pi_{1|A} : A  \rightarrow \PP^1$ whose restriction to $L = A_{\mathrm{red}}$ is the isomorphism $\pi_{1|L}: L\rightarrow \PP^1$. Thus $A$ is a split ribbon (see \cite[Corollary 1.7]{be}). Thus $\Oo_A$ fits into an exact sequence
\begin{equation}\label{eqcc1}
0 \to \Oo _L(a-2) \to \Oo _A \to \Oo _L\to 0,
\end{equation}
which splits as an exact sequence of $\Oo _L$-modules. 

Since $L$ is a complete intersection in $X$ of two planes of type $(0,1,0), (0,0,1)\in \Pic (X)$, the Koszul resolution for $\Ii_L$ shows that $\Ii_L/\Ii_L^2 \cong \Oo_L^{\oplus 2}$. In particular the normal bundle $N_{L|X}$ is trivial and double lines supported on $L$ are parametrized by the surjective morphism in $\mathrm{Hom}(\Oo_L^{\oplus 2}, \Oo_L(a-2))$ as in \cite[Introduction]{M} and \cite[Proposition 1.4]{N}. It proves part (1). The part (2) follows because two disjoint lines have genus $-1$. 

For the part (3), note that the maps split when $a=2$, showing that $A$ is a complete intersection. Let $\EE$ be the subset of $\mathbf{H}(2,0,0,2)_+$ parametrizing two disjoint lines of tridegree $(1,0,0)$. Since $h^1(N_C)=0$ for every $[C]\in \EE$, we have $\mathbf{H}(2,0,0,1)_+$ is smooth at each point of $\EE$ and of dimension $h^0(N_C) =4$. Obviously $\mathbf{H}(2,0,0,a)_{+,\mathrm{red}} = \EE \cup \Dd _2$, and $\EE$ is irreducible and rational. We have $\dim (\Dd_2)= 3$ and so it is sufficient to prove that at each $[B]\in \Dd _2$ the scheme $\mathbf{H}(2,0,0,a)_+$ is smooth and it has dimension $4$. Fix $[B]\in \Dd _2$. Since $B$ has only planar singularities, it is locally unobstructed (see \cite[2.12.1]{kol}). Hence $\mathbf{H}(2,0,0,a)_+$ has dimension at least $\chi (N_B)$ at $[B]$ by \cite[Theorem 2.15.3]{kol}). Since $\Dd _2$ is irreducible, we get $\chi (N_A) = \chi (N_E)$ for all $[A],[E]\in \Dd _2$. If $B$ is contained in a smooth quadric surface $T\in |\Oo _X(1,1,0)|$, then $N_B \cong \Oo _B^{\oplus 2}$ and so $\chi (N_B) = 4$.  In this case we also have $h^1(N_B) =0$ by (\ref{eqcc1}).

\quad {\emph {Claim}}: There is a connected $0$-dimensional subscheme $Z\subset \PP^1\times \PP^1$ of degree $2$ such that $B \cong L\times Z$.

\quad {\emph {Proof of Claim}}: For each connected $0$-dimensional subscheme $Z\subset \PP^1\times \PP^1$ of degree $2$, we have $L\times Z\in \Dd _2$. Since the set of all such $Z$ is smooth, irreducible, complete and $3$-dimensional, we get the assertion. \qed

There are two types of connected $0$-dimensional subscheme $Z\subset \PP^1\times \PP^1$ of degree $2$: the ones are contained in a ruling of $\PP^1\times \PP^1$ and the other ones are the complete intersection of two elements of $|\Oo _{\PP^1\times \PP^1}(1,1)|$. Thus our double lines $B$'s are the complete intersection of two elements of $|\Oo _X(0,1,1)|$. Hence even for these $B$'s we have $h^1(N_B)=0$, concluding the proof of the smoothness of $\mathbf{H}(2,0,0,a)_+$.

Obviously Claim holds also for the reduced $[C]\in \mathbf{H}(2,0,0,a)_+$ with $Z$ a reduced subscheme of $\PP^1\times \PP^1$ of degree $2$. Thus we get $\mathbf{H}(2,0,0,a)_+ \cong \mathbf{Hilb}^2(\PP^1\times \PP^1)$, which is isomorphic to the blow-up of the symmetric product $\mathrm{Sym}^2(\PP^1\times \PP^1)$ along the diagonal.
\end{proof}

\begin{remark}
From the proof of Theorem \ref{o1} each double line in $\Dd_a$ is associated to the triple $(L,f,g)$ with $L\subset X$ a line of tridegree $(1,0,0)$ and $f,g\in \CC[x_0, x_1]_{a-2}$ with no common zero, where $x_0$ and $x_1$ are homogeneous linear forms on $L$. 
\end{remark}

\begin{remark}
Fix an integer $a\ge 2$ and any $[A]\in \Dd _a$. For $(u,v) \in \ZZ^{\oplus 2}$, we have 
\begin{equation}\label{cccc}
\begin{tabular}{|c||c|c|c|}
  \hline
   & $h^0(\Oo_A(t,u,v))$ & $h^1(\Oo_A(t,u,v))$ & $h^2(\Ii_A(t,t,t))$ \\
  \hline
  \hline
  $ -1\le t $ & $t+a$ & $0$ & $0$  \\
  \hline
  $-a+1\le t\le -2$ & $t+a-1$ & $-t-1$ & $-t-1$  \\
  \hline
 $t\le -a$ & $0$ & $-2t-a$ & $-2t-a$ \\
  \hline
\end{tabular}
\end{equation}
Indeed we have $h^2(\Ii _A(t,t,t)) = h^1(\Oo _A(t,t,t))$ for all $t$, since $h^1(\Oo _X(t,t,t)) = h^2(\Oo _X(t,t,t)) =0$. Now the table follows from (\ref{eqcc1}) by setting $L:= A_{\mathrm{red}}$, because it is a split exact sequence as $\Oo _L$-sheaves. 

Since $h^1(\Oo _L(a-2)) =0$, the exponential sequence associated to (\ref{eqcc1})
$$0\to \Oo_L(a-2) \to \Oo_A^* \to \Oo_L^* \to 1$$
gives that the restriction map $\mathrm{Pic}({A}) \rightarrow \mathrm{Pic}(L)$ is bijective, as in the proof of \cite[Proposition 4.1]{be}. Thus we have $\Oo _A(c,u,v) \cong \Oo _A(c,c,c)$ for all $(c,u,v)\in \ZZ^{\oplus 3}$ and $\Oo _A(c,c,c)$ is the only line bundle on $A$ whose restriction to $L$ is $\Oo_L(c,c,c)$. Hence the table (\ref{cccc}) computes the cohomology groups of all line bundles on $A$.
\end{remark}

\begin{remark}
Take $[A]\in \Dd _2$. We saw in the proof of part (3) in Theorem \ref{o1} that $A$ is either the complete intersection of two elements of $|\Oo _X(0,1,1)|$, the case in which $Z$ is not contained in a ruling of $\PP^1\times \PP^1$, or a complete intersection of an element of $|\Oo _X(0,0,1)|$ (resp. $|\Oo _X(0,1,0)|$) and an element of $|\Oo _X(0,2,0)|$ (resp. $|\Oo _X(0,0,2)|$), the case in which $Z$ is contained in a ruling of $\PP^1\times \PP^1$. In both cases we have $h^1(\Ii _A(t,t,t)) =0$ for all $t\ne 0$ and $h^1(\Ii _A)=1$.
\end{remark}

\begin{lemma}\label{bp02}
For $a\ge 2$, fix $[A]\in \Dd_a$ with $L:=A_{\mathrm{red}}$. 
\begin{itemize}
\item[(i)] For a line $L'\subset X$ different from $L$, we have $\deg (A\cap L') \le 2$. Moreover we have $A\cap L'=\emptyset$ if and only if $L\cap L' =\emptyset$. 
\item[(ii)] The following set is a non-empty irreducible and rationally connected variety of dimension $2a+1;$
$$\left\{(B,L')~|~[B]\in \Dd_a ,~ L' \text{ a line of tridegree $(0,1,0)$ with }B\cap L'=\emptyset \right\}.$$
\item[(iii)] The tangent plane $T_pA$ at a point $p\in L$ is a plane containing $L$ and contained in the $3$-dimensional tangent space $T_pX$. 
\item[(iv)] $T_pA\cap X$ is the union of the three lines $L, L_1, L_2$ of $X$ through $p$, where $L_1$ of tridegree $(0,1,0)$ and $L_2$ of tridegree $(0,0,1)$. We have $\deg (A\cap L_i) =2$ if and only if $T_pA$ is the plane spanned by $L\cup L_i$.
\end{itemize}
\end{lemma}

\begin{proof}
Since $L$ and $L'$ are different, so $\deg (A\cap L')$ is a well-defined non-negative integer, and we have $\deg (A\cap L')=0$ if and only if $A\cap L' =\emptyset$, i.e. $L\cap L' =\emptyset$. Since $L$ has tridegree $(1,0,0)$, there is $(o,o')\in \PP^1\times \PP^1$ such that $L = \PP^1\times \{(o,o')\}$. The complement of $\PP^1\times \{o'\}$ in $\PP^1 \times \PP^1$ parametrizes the set of all lines $T$ of tridegree $(0,1,0)$ with $T\cap L=\emptyset$. The other assertions are obvious.
\end{proof}

\begin{lemma}\label{bo01}
For $[A]\in \mathbf{H}(2,0,0,2)_+$, there exists a line $R \subset X$ of tridegree $(0,1,0)$ with $\deg (A\cap R)\ge 2$, i.e. $\deg (A\cap R)=2$ if and only if there exists $Q\in |\Oo _X(0,0,1)|$ that contains $A$. In this case $Q$ is unique, $A\cup R\subset Q$ and there is a $1$-dimensional family of such lines $R$.
\end{lemma}

\begin{proof}
The lemma is obvious if $A$ is a disjoint union of two lines, say $A = \PP^1\times \{(o_2, o_3)\}\cup \PP^1\times \{(p_2, p_3)\}$, because the existence of $R$ is equivalent to $o_3=p_3$. Now assume $[A]\in \Dd _2$, say associated to $(L,f,g)$ with $L:= A_{\mathrm{red}}$ and $(f,g)\in \CC^2\setminus \{(0,0)\}$.  Write $L =\PP^1\times \{(p_2,p_3)\}$.
Since $\PP^1\times \PP^1\times \{p_3\}$ is the only element of $|\Oo _X(0,0,1)|$ containing $L$, the uniqueness part is obvious. Assume the existence of a line $ R \subset X$ of tridegree $(0,1,0)$ with $\deg (A\cap R)\ge 2$. By Lemma \ref{bp02} we have $\deg (A\cap R) =2$ and $R\cap L$ contains a point, say $p = (p_1,p_2,p_3)$. For each point $q = (q_1,p_2,p_3)\in L$ the pull-backs via the projections $\pi _i$ for $i=2,3$, of a non-zero tangent vector of $\PP^1$ at $p_i$ form a basis of $N_{L,q} \cong \CC^2$. Since $A$ has tridegree $(2,0,0)$, so the map $\pi _{2|A}$ is induced by an element of $H^0(\Oo _A)$, i.e. by an element $c\in H^0(\Oo_L)$ due to (\ref{eqcc1}) with $a=2$ and $\pi_2(L)=\{p_2\}$. The condition $\deg (A\cap R) =2$ is equivalent to saying that $\pi _{2|A}$ vanishes at $p$. Since $c$ is a constant, $\pi _{2|A}$ vanishes at all points of $L$, i.e. $A \subset \PP^1\times \PP^1\times \{p_3\}$.
\end{proof}


\section{Irreducibility of Hilbert schemes}

\begin{theorem}\label{bo02}
We have $$\mathbf{H}(2,1,0,1)_+ \cong \PP^5\times \PP^1.$$
\end{theorem}

\begin{proof}
It is sufficient to prove that for each $[C]\in \mathbf{H}(2,1,0,1)_+$ there is $Q\in |\Oo _X(0,0,1)|$ such that $C\in |\Oo _Q(1,2)|$, which would give us a morphism from $\mathbf{H}(2,1,0,1)_+ $ to $\PP^5 \times \PP^1$. Its inverse map is obviously defined. If $C$ is reduced, then $\pi_{2|C}$ shows that each irreducible component of $C$ is smooth and rational. Since $\chi (\Oo _C)=1$, so $C$ is connected. Since $C$ is reduced, connected and of tridegree $(2,1,0)$, the scheme $\pi _3({C})$ is a point and so there is a point $o\in \PP^1$ such that $C \subset \PP^1\times \PP^1\times \{o\}$. Now assume that $C$ is not reduced. By Lemma \ref{a1} and Theorem \ref{o1} we see that $C = A\cup R$ with $[A]\in \Dd _a$ for $a\ge 2$ and $R$ a line. Since $\deg (A\cap R) \le 2$ and $\chi (\Oo _C) =\chi (\Oo _A)+1-\deg (A\cap R)$, so we get $a=2$ and $\deg (A\cap R) =2$. We may now apply Lemma \ref{bo01}.
\end{proof}

\begin{remark}\label{bo03}
The proof of Theorem \ref{bo02} shows that for each $[C]\in \mathbf{H}(2,1,0,1)_+$ there is $Q\in |\Oo _X(0,0,1)|$ such that $C\in |\Oo _Q(1,2)|$. Hence we have $h^i(\Ii _C(u,v,w)) = h^i(\Ii _{C'}(u,v,w))$ for all $i\in \{1,2,3\}$, $(u,v,w)\in \ZZ^{\oplus 3}$ and $[C],[C']\in \mathbf{H}(2,1,0,1)_+$.
\end{remark}

\begin{lemma}\label{bo04}
For a fixed $[A]\in \Dd _a$ with $a\ge 3$, define $\Ss$ to be the set of all lines $R\subset X$ with tridegree $(0,0,1)$ and $\deg (A\cap R)\ge 2$. Then $\Ss$ is a non-empty finite set and $\deg (A\cap R) =2$ for all $R\in \Ss$.
\end{lemma}

\begin{proof}
Set $L:= A_{\mathrm{red}}$. We have $\chi (\Oo _B) =2<a$ for every $B\in |\Oo _{\PP^1\times \PP^1}(0,2)|$ and so $\pi_{12|A}: A \rightarrow \PP^1\times \PP^1$ is not an embedding by \cite[Proposition II.2.3]{Hartshorne}. Thus we have $\Ss \ne \emptyset$. Fix $R\in \Ss$. Lemma \ref{bp02} gives $\deg (A\cap R)=2$ and so there is a unique point $o\in L\cap R$. $R$ is the unique line of tridegree $(0,0,1)$ containing $o$. The condition $\deg (A\cap R)=2$ is equivalent to the condition that the plane $\langle L\cup R\rangle$ is the tangent plane of $A$ at $o$. Hence we have $\deg (A\cap L_1) =1$ only for the line $L_1$ of tridegree $(0,1,0)$ containing $o$. Set $Q:= \pi _{12}(L)\times \PP^1\in |\Oo _X(0,1,0)|$. Assume that $\Ss$ is infinite. We get that $Q$ contains infinitely many tangent planes of $A$ and so each tangent plane of $A$ is contained in $Q$. Therefore we have $\deg (A\cap T)\le 1$ for all lines $T\subset X$ of tridegree $(0,1,0)$ and so $\pi_{13|A}: A\rightarrow \PP^1\times \PP^1$ is an embedding by \cite[Proposition II.2.3]{Hartshorne}. We saw that this is false.
\end{proof}

\begin{remark}\label{zz1}
Let us fix a double line $[A]\in \Dd_a$ with $a\ge 3$, that is associated to the triple $(L,f,g)$ with $L\subset X$ a line of tridegree $(1,0,0)$ and $f,g\in \CC[x_0,x_1]_{a-2}$ with no common zero. For a fixed point $p =(o_1,o_2,o_3)\in L$ and the line $R\subset X$ of tridegree $(0,1,0)$ passing through $p$, we have $1\le \deg (R\cap A)\le 2$. Indeed we have $\deg (R\cap A)=2$ if and only if $f$ vanishes at $p$. Since $a>2$, there exists at least one line $R$ with this property and at most
$(a-2)$ such lines exist. If $f$ is general (and in particular if $A$ is general), then $f$ has $(a-2)$ distinct zeros and so there are exactly $(a-2)$ lines $R$ of tridegree $(0,1,0)$ with $\deg (A \cap R)=2$.  
\end{remark}

\begin{theorem}\label{bo05}
For each integer $a\ge 4$, we have
$$\mathbf{H}(2,1,0,a)_{+,\mathrm{red}} =\Ss _0\cup \Ss _1 \cup \Ss _2,$$
where we have $[C]\in \Ss _i$  for $i\in \{0,1,2\}$ if and only if $C = A\cup R$ where $[A]\in \Dd _{a+i-1}$ and $R$ is a line of tridegree $(0,1,0)$ with $\deg (A\cap R)=i$. Furthermore we have
\begin{itemize}
\item $\overline{\Ss_0}$ is an irreducible component of $\mathbf{H}(2,1,0,a)_{+, \mathrm{red}}$ with $\dim (\Ss _0) =2a-1$;
\item $\Ss_2$ is irreducible with $\dim (\Ss _2) = 2a+1$;
\item $\Ss_1$ is irreducible with $\dim (\Ss _1) = 2a$ and $\Ss _1\subset \overline{\Ss _2}$;
\item $\overline{\Ss_0}$ and $\overline{\Ss _2}$ are the irreducible components of $\mathbf{H}(2,1,0,a)_{+,\mathrm{red}}$.
\end{itemize}
\end{theorem}

\begin{proof}
Fix $[C]\in \mathbf{H}(2,1,0,a)_{+,\mathrm{red}}$ and then $C$ is not reduced, because we assumed $a \ge 4$. By Lemma \ref{a1} and Theorem \ref{o1} we have $C = A\cup R$ where $[A]\in \Dd _c$ for some $c\ge 2$ and $R$ is a line of tridegree $(0,1,0)$ with $\deg (A\cap R) = c+1-a$. Since
$0 \le \deg (A\cap R) \le 2$, we have $c \in \{a-1,a,a+1\}$ and so we get a set-theoretic decomposition $\mathbf{H}(2,1,0,1)_{+,\mathrm{red}} =\Ss _0\cup \Ss _1 \cup \Ss _2$. 

Now we check that $\Ss _0$ and $\Ss _1$ are irreducible. For $[A]\in \Dd _{a-1}$, the set of all lines $R$ with $R\cap A =\emptyset$ is a non-empty open subset of $\PP^1\times \PP^1$. By Theorem \ref{o1}, $\Ss _0$ is non-empty, irreducible, rational and of dimension $2(a-1)-1+2=2a-1$. For $[A]\in \Dd _a$, set $L:= A_{\mathrm{red}}$ and then a line $R$ of tridegree $(0,1,0)$ satisfies $\deg (A\cap R)>0$ if and only if $R\cap L \ne \emptyset$. By Lemma \ref{bo04} the set of all such lines $R$ is a non-empty smooth rational curve. Hence $\Ss _1$ is rationally connected, irreducible and of dimension $2a$. Since the support of each element of $\Ss _1\cup \Ss _2$ is connected, we have
$\Ss _0\nsubseteq \overline{\Ss _1\cup \Ss _2}$, and so $\overline{\Ss _0}$ is an irreducible component of $\mathbf{H}(2,1,0,a)_{+,\mathrm{red}}$. 

For $[A]\in \Dd _{a+1}$, Remark \ref{zz1} shows that the set of all lines of tridegree $(0,1,0)$ with $\deg (A\cap R)=2$ is non-empty and finite. So we get $\Ss _2 \ne \emptyset$ and each irreducible component of $\Ss _2$ has dimension $2a+1$.

\quad {\emph {Claim}}: $\Ss _2$ is irreducible.

\quad {\emph {Proof of Claim}}: Let $\II$ be the set of all pairs $(A,p)$ with $[A]\in \Dd _{a+1}$, $p\in A_{\mathrm{red}}$ and there is a line $R\subset X$ of tridegree $(0,1,0)$ with $p\in R$ and $\deg (R\cap A)=2$. Then it is sufficient to prove that $\II$ is irreducible. For a fixed line $L\subset X$ of tridegree $(1,0,0)$ and $g\in \CC[x_0, x_1]_{a-1}$ with $g\ne 0$, we define $U(L,g)$ be the set of all $[A]\in \Dd_{a+1}$ associated to a triple $(L,f,g)$ for some $f$, and let $\II_{L,g}$ be the set of all pairs $(A,p)$ with $[A]\in U(L,g)$, $p\in L$ and there is a line $R\subset X$ of tridegree $(0,1,0)$ with $\deg (R\cap A)=2$. The irreducibility of $\II _{L,g}$ is equivalent to the well-known irreducibility of the set of all pairs $(f,p)$ with $p\in \PP^1$, $f\in \CC[x_0,x_1]_{a-1}\setminus \{0\}$ vanishing at $p$. Thus $\II$ is irreducible and so is $\Ss_2$. \qed

Now it remains to prove that $\Ss_1\subset \overline{\Ss _2}$. Fix a general $A\cup R \in \Ss _1$ with $A$ associated to a triple $(L,f,g)$ and set $\{p\}:= R\cap L$. Since $\deg (R\cap A) =1$, we have $f({p})\ne 0$. For a general $A\cup R$ we may assume that $g({p}) \ne 0$. Let $S\subset L$ be a finite set containing all zeros of $f$ and $g$ and the point $x_0=0$, but with $p\notin S$. Set $z:= x_1/x_0$ and let $f_1,g_1$ the elements of $\CC[z]$ obtained by dehomogenizing $f$ and $g$. For a general $A$ we may assume that $\deg (f_1)=\deg (g_1) =a-2$. Let $\Delta$ denote the diagonal of $(L\setminus S)\times (L\setminus S)$. For all $(u,v)\in ((L\setminus S)\times (L\setminus S))\setminus \Delta$, set $f_{u,v}:= (z-u)f_1$ and $g_{u,v}:= (z-v)g_1$. Let $\tilde{f}_{u,v}(x_0,x_1)$ (resp. $\tilde{g}_{u,v}(x_0,x_1)$) be the homogeneous polynomial associated to $f_{u,v}$ (resp. $g_{u,v}$). Let $A_{u,v}$ denote the element of $\Dd _{a+1}$ associated to $(L,\tilde{f}_{u,v},\tilde{g}_{u,v})$ and let $R_u$ the line of tridegree $(0,1,0)$ passing through the point of $L$ associated to $u$. We got a flat family $\{A_{u,v}\cup R_u\}_{u\in \Delta}$ of elements of $\Ss _2$. As $(u,v)$ tends to $(p,p)$, we get that $A\cup R$ is a flat limit of this family.
\end{proof}

\begin{theorem}\label{bo06} 
We have
$$\mathbf{H}(2,1,0,3)_{+,\mathrm{red}} =\Tt  \cup \Tt_1\cup \Tt_2,$$
where each curve $[C]\in \Tt_i$ for $i\in \{1,2\}$ is of the form $A \cup R$ where $[A]\in \Dd_{i+2}$ and $R$ is a line of tridegree $(0,1,0)$ with $\deg (A \cap R)=i$. A general element of $\Tt$ is a disjoint union of three lines. Furthermore we have
\begin{itemize}
\item $\overline{\Tt}$ is an irreducible component of $\mathbf{H}(2,1,0,3)_{+, \mathrm{red}}$ with $\dim (\Tt)=6$; 
\item $\Tt_2$ is an irreducible with $\dim (\Tt_2)=7$;
\item $\Tt_1$ is irreducible with $\dim (\Tt_1)=6$ and $\Tt_1 \subset \overline{\Tt_2}$; 
\item $\overline{\Tt}$ and $\overline{\Tt_2}$ are the irreducible components of $\mathbf{H}(2,1,0,3)_{+, \mathrm{red}}$. 
\end{itemize}
\end{theorem} 

\begin{proof}
Fix $[C]\in \mathbf{H}(2,1,0,3)_{+,\mathrm{red}}$ and then $C$ is reduced if and only if it is the disjoint union of three lines, two of tridegree $(1,0,0)$ and one of tridegree $(0,1,0)$. The set $\Aa$ of all such curves is a non-empty open subset of $\PP^2\times \PP^2\times \PP^2$ and so $\Aa$ is smooth, irreducible and rational with $\dim (\Aa) =6$. Note that $\Aa$ is not complete since $\Aa \subsetneq \PP^2\times \PP^2\times \PP^2$. 

Now assume that $C$ is not reduced.  By Lemma \ref{a1} and Theorem \ref{o1} we have
$C = A\cup R$ where $[A]\in \Dd _c$ with $c\ge 2$ and $R$ is a line of tridegree $(0,1,0)$ with $\deg (A\cap R) = c-2$. Since $0 \le \deg (A\cap R) \le 2$, so we have $c\in \{2,3,4\}$. Let $\Tt_i$ be the set of all $C =A\cup R$ with $\deg (A\cap R) =i$. As in the proof of Theorem \ref{bo05} we see that $\Tt_i\ne \emptyset$ for all $i$. Set $\Tt:= \Aa \cup \Tt_0$ to be the set of all disjoint unions of an element of $\mathbf{H}(2,0,0,2)_{+,\mathrm{red}}$ and a line of tridegree $(0,1,0)$. Since the reduction of each element of $\Tt_1 \cup \Tt_2$ is connected, so we have $\Tt\not\subseteq \overline{\Tt_1 \cup \Tt_2}$. By the case $a=2$ of Theorem \ref{o1}, $\Tt_0$ is in the closure of $\Aa$ and so $\Tt$ is irreducible. As in the proof of Theorem \ref{bo05}, we get that $\Tt_2$ is irreducible and $\Tt_1 \subset \overline{\Tt_2}$. 
\end{proof}

In $\mathbf{H}(1,1,1,1)$ we have a family of curves formed by three lines through a common point. Denote the locus of such curves by $\mathbf{D}$ and we have $\mathbf{D} \cong X$. 

\begin{lemma}\label{bb3}
$\mathbf{H}(1,1,1,3)_{+,\mathrm{red}}$ is irreducible, smooth and rational of dimension $6$.
\end{lemma}
\begin{proof}
Fix a curve $[C]\in \mathbf{H}(1,1,1,3)_+$ and then $C$ is reduced by Lemma \ref{a1}. We also have $\chi (\Oo _C)=-2$ and so $C$ has at least three connected components. Thus $C$ is a disjoint union of three lines, one line for each tridegree $(1,0,0)$, $(0,1,0)$ and $(0,0,1)$. Hence set-theoretically $\mathbf{H}(1,1,1,3)_{+,\mathrm{red}}$ is irreducible, rational and of dimension $6$. Since $N_C\cong \Oo _C^{\oplus 2}$, we have $h^1(N_C) = 0$ and so $\mathbf{H}(1,1,1,3)_+$ is smooth.
\end{proof}

\begin{lemma}\label{bb4}
$\mathbf{H}(1,1,1,2)_+$ has three connected components and each of them is smooth and rational of dimension $7$.
\end{lemma}

\begin{proof}
Fix $[C]\in \mathbf{H}(1,1,1,2)_+$ and then $C$ is reduced again by Lemma \ref{a1}. We have $\chi (\Oo _C)=-1$ and so $C$ has at least two connected components. One of these connected components must be a line. Since $\chi (\Oo _C) \ne -2$ and $\deg ({C}) =3$, so $C$ is not the union of three disjoint lines. Hence $C$ has a unique connected component of degree $1$. The three connected components of $\mathbf{H}(1,1,1,2)_+$ are distinguished by the tridegree of their degree $1$ component. 

With no loss of generality we may assume that $C$ has a line $L$ of tridegree $(1,0,0)$ as a connected component, say $C = L\sqcup D$ with $D$ of tridegree $(0,1,1)$. We have $D = \{o\}\times D'$ for a point $o\in \PP^1$ and a conic $D'\in |\Oo _{\PP^1\times \PP^1}(1,1)|$. Since $|\Oo _{\PP^1\times \PP^1}(1,1)|$ is irreducible and of dimension $3$, we get that each connected component of $\mathbf{H}(1,1,1,2)_{+,\mathrm{red}}$ is irreducible, rational and of dimension $7$. Since $N_L \cong \Oo _L^{\oplus 2}$ and $N_D \cong \Oo _D\oplus \Oo _D(1,1)$ with $L\cap D=\emptyset$, so we get $h^1(N_C) =0$ and so $\mathbf{H}(1,1,1,2)_+$ is smooth.
\end{proof}

\begin{proposition}\label{a2}
$\mathbf{H}(1,1,1,1)$ is irreducible, unirational of dimension $6$ and smooth outside $\mathbf{D}$. 
\end{proposition}
\begin{proof}
Let us fix $[C]\in \mathbf{H}(1,1,1,1)$. By Lemma \ref{a1} every $1$-dimensional component of $C$ is generically reduced, i.e. the purely $1$-dimensional subscheme $E$ of $C_{\mathrm{red}}$ has tridegree $(1,1,1)$. We have $\chi (\Oo _C) \ge \chi (\Oo _D)$ for each connected component $D$ of $E$ and equality holds if and only if $D =C$. Since we have $\chi (\Oo _D) \ge 1$, we get $C = D$, and so $C$ is connected and reduced.

\quad (a) First assume that $C$ is irreducible. Since $\pi_{1|C}: C \rightarrow \PP^1$ has degree $1$, so $C$ is smooth and rational. In particular we get $[C]\in \mathbf{H}(1,1,1,1)^{\mathrm{sm}}$. Since $\pi_{i|C} :C \rightarrow \PP^1$ for $i=1,2,3$, is induced by the complete linear system $|\Oo _{\PP^1}(1)|$, so $\mathbf{H}(1,1,1,1)^{\mathrm{sm}}$ is homogeneous for the action of the group 
$$\mathrm{Aut}^0(X) =PGL(2)\times PGL(2)\times PGL(2).$$
Thus the algebraic set $\mathbf{H}(1,1,1,1)^{\mathrm{sm}}_{\mathrm{red}}$ is irreducible and unirational. To show that $\mathbf{H}(1,1,1,1)^{\mathrm{sm}}$ is smooth and of dimension $6$, it is sufficient to prove that $h^1(N_C)=0$ and $h^0(N_C) =6$. Note that we have $\chi (N_C) =6$. Since $X$ is homogeneous, $TX$ is globally generated and so is $TX_{|C}$. Since $N_C$ is a quotient of $TX_{|C}$, so $N_C$ is also globally generated. Since $C\cong \PP^1$, we get $h^1(N_C)=0$. Indeed we have $N_C \cong \Oo_{\PP^1}(2)\oplus \Oo_{\PP^1}(2)$; The normal bundle $N_C$ is a direct sum of two line bundles, say of degree $z_1\ge z_2$ with $z_1+z_2 =4$. Since $N_C$ is a quotient of $TX_{|C}$, which is the direct sum of three line bundles of degree $2$, we get $z_1=z_2 =2$.

\quad (b) By part (a) it is sufficient to prove that $\mathbf{H}(1,1,1,1)$ is smooth at each reducible curve $[C]\not \in \mathbf{D}$ and that each reducible element of $\mathbf{H}(1,1,1,1)$ is in the closure of $\mathbf{H}(1,1,1,1)^{\mathrm{sm}}$.

\quad ({b1}) Now assume that $C$ has two irreducible components, say $C = D_1\cup D_2$ with $D_1$ a line. Since $\chi (\Oo _C) = 1$ and the scheme $C$ is reduced with no isolated point and arithmetic genus $0$, so it is connected. Since $p_a({C})=0$, we get $\deg (D_1\cap D_2) =1$.
In particular $C$ is nodal and so $N_C$ is locally free. Without loss of generality we may assume that $D_1$ has tridegree $(1,0,0)$ and so $D_2$ is a smooth conic with tridegree $(0,1,1)$. We have $N_{D_1} \cong \Oo _{D_1}^{\oplus 2}$ and
$N_{D_2} \cong \Oo _{D_2}\oplus \Oo _{D_2}(0,1,1)$. Since $N_{C}$ is locally free, we have a Mayer-Vietoris exact sequence
\begin{equation}\label{eqa2}
0 \to N_C \to N_{C_{|D_1}} \oplus N_{C_{|D_2}} \to N_{C_{|D_1\cap D_2}}\to 0.
\end{equation}
Since $\deg (D_1\cap D_2)=1$ and $C$ is nodal, the sheaf $N_{C_{|D_1}}$ (resp. $N_{C_{|D_2}}$) is a vector bundle of rank $2$ obtained from $N_{D_1}$ (resp. $N_{D_2}$) by making one positive elementary transformation at $D_1\cap D_2$ (see \cite[\S 2]{HH}, \cite[Lemma 5.1]{S}, \cite{Sernesi}), i.e. $N_{D_i}$ is a subsheaf of $N_{C_{|D_i}}$ and its quotient $N_{C_{|D_i}}/N_{D_i}$ is
a skyscraper sheaf of degree $1$ supported on the point $D_1\cap D_2$.  Since $h^1(D_2,N_{D_2})=0$, we get $h^1(N_{C_{|D_2}})=0$. We also get $h^1(D_1,N_{C_{|D_1}}) =0$ and that $N_{C_{|D_1}}$ is spanned. Since $\deg (D_1\cap D_2)=1$ and $N_{C_{|D_1}}$ is spanned, so the restriction map $H^0(D_1,N_{C_{|D_1}})\rightarrow H^0(D_1\cap D_2,N_{C_{|D_1\cap D_2}})$ is surjective. Thus (\ref{eqa2}) gives $h^1(N_C)=0$ and so $\mathbf{H}(1,1,1,1)$ is smooth of dimension $6$ at $[C]$. Since the set of all such curves $C$ has dimension $5$, so $[C]$ is in the closure of $\mathbf{H}(1,1,1,1)^{\mathrm{sm}}$.

\quad ({b2}) Now assume that $C$ has at least three components, i.e. $C = D_1\cup D_2\cup D_3$ with each $D_i$ a line. First assume that $C$ is nodal. In this case one of the lines, say $D_2$, meets the other lines. As in step (a) we first get $h^1(N_{D_2\cup D_3}) =0$ and then $h^1(N_{C}) =0$. Thus $\mathbf{H}(1,1,1,1)$ is smooth of dimension $6$ at $[C]$. Since the set of all such $C$ has dimension $4$, we get that $[C]$ is in the closure of $\mathbf{H}(1,1,1,1)^{\mathrm{sm}}$. Now assume that $[C]\in \mathbf{D}$, say $C = A_1\cup A_2\cup A_3$ with a common point $p$. We can deform $A_1$ in a family of lines intersecting $A_2$ at a point different from $p$ and not intersecting $A_3$. Thus even in this case $[C]$ is in the closure of $\mathbf{H}(1,1,1,1)^{\mathrm{sm}}$.
\end{proof}

\begin{lemma}\label{2.1.1}
$\mathbf{H}(2,1,1,1)^{\mathrm{sm}}$ is irreducible, unirational and smooth with dimension $8$. Each $[C]\in \mathbf{H}(2,1,1,1)^{\mathrm{sm}}$ is a smooth and connected rational curve with tridegree $(2,1,1)$.
\end{lemma}
\begin{proof}
Since no plane cubic curve is contained in $X$ and the intersection of two quadric surfaces in $\PP^3$ has $4t$ as its Hilbert polynomial, $[C]\in \mathbf{H}(2,1,1,1)^{\mathrm{sm}}$ is a quartic rational curve. By Remark \ref{tr1}, a general element in 
$$V=H^0(\Oo_{\PP^1}(2))^{\oplus 2} \times H^0(\Oo_{\PP^1}(1))^{\oplus 2} \times H^0(\Oo_{\PP^1}(1))^{\oplus 2}$$
gives an isomorphism $\alpha : \PP^1 \rightarrow X$ onto its image and so there is an open subset $V_0\subset V$ with the universal family $\mathcal{V}_0\subset V_0\times X$. Since $\mathcal{V}_0$ is flat, so it gives a surjection $Y_0 \rightarrow \mathbf{H}(2,1,1,1)^{\mathrm{sm}}$ by the universal property of the Hilbert scheme. Since $V_0$ is rational, so $\mathbf{H}(2,1,1,1)^{\mathrm{sm}}$ is unirational. Now fix a curve $[C]\in \mathbf{H}(2,1,1,1)^{\mathrm{sm}}$. Since $C$ is a twisted cubic curve, so by adjunction we have 
$$\Oo_{\PP^1}(-2)\cong \omega_C \cong \det (N_{C}) \otimes \Oo_X(-2,-2,-2)$$
and so $\det (N_{C})\cong \Oo_{\PP^1}(6)$. It implies that $N_{C}\cong \Oo_{\PP^1}(a)\oplus \Oo_{\PP^1}(b)$ with $a+b=6$. Now from the surjection $TX_{\vert_C} \rightarrow N_{C}$, we get that $N_{C}$ is globally generated and so $a,b\ge 0$. In particular we have $h^0(N_{C})=8$ and $h^1(N_{C})=0$. Indeed we have $N_C \cong \Oo_{\PP^1}(4)\oplus \Oo_{\PP^1}(2)$. 
\end{proof}

Let us write $\mathbf{H}(2,1,1,1)_{+, \mathrm{red}}=\Gamma_1 \sqcup \Gamma_2$, where $\Gamma_1$ consists of the reduced curves and $\Gamma_2$ consists of the non-reduced curves.

\begin{lemma}\label{2.1.1.0}
$\Gamma_1$ is irreducible with $\mathbf{H}(2,1,1,1)^{\mathrm{sm}}$ as its open subset. In particular, each $[C]\in \Gamma _1$ is connected and irreducible components of $C$ are smooth rational curves. Furthermore we have $h^1(\Ii _C(t)) =0$ for all $t\in \ZZ$.
\end{lemma}

\begin{proof}
For a fixed curve $[C]\in \Gamma _1$, let $T$ be any irreducible component of $C$. Since either $T$ is a fiber of $\pi _{12}$ or ${\pi _3}_{|T}$ has degree one, so $T$ is smooth and rational. Assume for the moment the existence of a connected curve $C'\subseteq C$ with $p_a(C') >0$. Since $p_a({C}) =0$, we have $\deg (C') \le 3$. Since $C'$ is reduced, we get that $C'$ is a plane cubic, contradicting the fact that $X$ contains no plane and it is cut out by quadrics in $\PP^7$. The non-existence of $C'$ implies that $C$ is connected. Let $C_1,\dots ,C_h$ be the irreducible components of $C$ with $h\ge 2$ in an ordering so that if $h\ge 3$, then $E_i:= C_1\cup \cdots \cup C_i$ is connected for all $2 \le i \le h-1$. Fix an integer $i$ with $1\le i \le h-1$. Since $E_i$ and $E_{i+1}$ are connected with arithmetic genus zero, we have $\deg (C_{i+1}\cap E_i) =1$. Since $TX \cong \Oo _X(2,0,0)\oplus \Oo _X(0,2,0)\oplus \Oo _X(0,0,2)$, so the Mayer-Vietoris exact sequence
$$0 \to TX_{|E_{i+1}} \to TX_{|E_i}\oplus TX_{|C_{i+1}} \to TX_{|E_i\cap C_{i+1}}\to 0$$and induction on $i$ give $h^1(TX_{|C}) =0$. Since the natural map $TX_{|C} \rightarrow N_C$ has cokernel supported on the finite set $\mathrm{Sing}({C})$, we have $h^1(N_C)=0$ and so $ \mathbf{H}(2,1,1,1)_+$ is smooth at $[C]$. If $C$ is nodal, by induction on $i$ we get that each $E_i$ is smoothable and in particular $C$ is smoothable in $X$, i.e. $[C]$ is contained in the closure of $\mathbf{H}(2,1,1,1)^{\mathrm{sm}}$ in $\mathbf{H}(2,1,1,1)$.

Now assume that $C$ is not nodal and then we get $3\le h\le 4$. If $h=3$ and we may find an ordering so that $\deg (C_1) =2$, $\deg (C_2)=\deg (C_3)=1$ and $C_1, C_2, C_3$ contain a common point, say $p$, and neither $C_2$ nor $C_3$ is the tangent line of $C_1$ at $p$. Since $p\in C_2\cap C_3$, the lines $C_2$ and $C_3$ have different tridegree and so we may fix $C_1\cup C_2$ and move $C_3$ in the family of all lines meeting $C_2$ and with the tridegree of $C_3$. Thus we may deform $C$ to a nodal curve and hence again we get that $[C]$ is contained in the closure of $\mathbf{H}(2,1,1,1)^{\mathrm{sm}}$ in $\mathbf{H}(2,1,1,1)$. If $h=4$, then each irreducible component of $C$ is a line. Since two of these components have the same tridegree, so $C$ has a unique triple point and we may use the argument above for the case $h=3$.

Now a Mayer-Vietoris exact sequence gives $h^0(\Oo _{E_i}(1)) = \deg (E_i)+1$ for all $i$ even in the non-nodal case. In particular $h^0(\Oo _C(1)) =5$. Let $M\subset \PP^7$ be the linear span of $C$. Since $h^0(\Oo _C(1)) =5$, we have $\dim (M) \le4$. Let $H\subset M$ be a general hyperplane. Assume for the moment $\dim (M) =4$. In this case $C$ is linearly normal in $M$. The scheme $H\cap C$ is the union of $4$ points. Since $C$ is connected and linearly normal in $M$, we have
$h^1(M,\Ii _C(t)) =0$ for all $t\le 1$. The case $t=1$ of the exact sequence
\begin{equation}\label{eq+t1}
0 \to \Ii _{C,M}(t-1) \to \Ii _{C,M}(t)\to \Ii _{C\cap H,H}(t)\to 0
\end{equation}
gives that $C\cap H$ is formed by $4$ points of $H$ spanning $H$. Hence $h^1(H,\Ii _{C\cap H,H}(t)) =0$ for all $t>0$. By induction on $t$, (\ref{eq+t1}) gives $h^1(M,\Ii _{C,M}(t)) =0$ for all $t\ge 2$. To conclude we only need to exclude that $\dim (M)<4$. We have $\dim (M)>2$, because $X$ is cut out by quadrics and contains no plane. Now assume $\dim (M) =3$. Since $X$ contains no plane and no quadric surface, $X\cap M$ is an algebraic set cut out by quadrics and with connected components of dimension at most $1$. $X\cap M$ is not the complete intersection of two quadrics of $M$, because $X\cap M$ contains the degree $4$ curve $C$ and $p_a({C})=0$. Since $h^0(M,\Oo _C) =1$, the case $t=1$ of (\ref{eq+t1}) gives that $C\cap H$ spans the plane $H$. Hence $h^1(H,\Ii _{C\cap H,H}(2)) =0$. Since $h^1(M,\Ii _{C,M}(1)) =1$ and $H\cap C$ is formed by $4$ points spanning the plane $H$, the case $t=2$ of (\ref{eq+t1}) gives $h^1(M,\Ii _{C,M}(2)) \le 1$ and hence $h^0(M,\Ii _{C,M}(2)) \le 2$, a contradiction.
\end{proof}




\begin{theorem}
$\mathbf{H}(2,1,1,1)_{+, \mathrm{red}}$ is irreducible. 
\end{theorem}
\begin{proof}
By Lemma \ref{2.1.1.0} it is enough to show that $\Gamma_2 \subset \overline{\Gamma_1}$. Fix $[C]\in \Gamma_2$, i.e. $C$ is not reduced. By Lemma \ref{a1} $C_{\mathrm{red}}$ has tridegree $(1,1,1)$ and the nilradical of $\Oo _C$ is supported by a line $L$ of tridegree $(1,0,0)$. There is a unique reduced curve $E\subset C$ with $E$ of tridegree $(0,1,1)$. $E$ is either a disjoint union of two lines or a reduced conic. Set $\Jj := \mathrm{Ann}_{\Oo _C}(\Ii _{E,C})$. The $\Oo _C$-sheaf $\Jj$ is the ideal sheaf of a degree $2$ structure supported by $L$, possibly with embedded components. Let $C'$ be the curve with $\Jj$ as its ideal sheaf and $A$ the maximal locally CM subcurve of $C'$, which is obtained by taking as its ideal sheaf in $C$ the intersection of the non-embedded components of a primary decomposition of $\Ii _{C'}$. The curve $A$ is a locally CM curve of degree $2$ with $L:=\PP^1$ as its support, i.e. $[A]\in \Dd _a$ for some $a\ge 2$. We have $\chi (\Oo _A) =a$.

\quad (a) Assume for the moment that $E$ is a reduced conic and let $\langle E\rangle$ be the plane spanned by $E$. Since $X$ contains no plane cubic, we have $\deg (L\cap \langle E\rangle )\le 1$. Considering a general hyperplane $H\subset \PP^7$ containing $\langle E\rangle$ with $L\nsubseteq H$, we get $\deg (A\cap \langle E\rangle )\le 2$ and so $\deg (A\cap E) \le 2$. Thus we have $\chi (\Oo _A) \ge a-1$ and so $a=2=\deg (A\cap E) $. It implies that $\deg (L\cap E) =1$. Set $Z:= A\cap E$. Since $E$ is connected, there is $o\in \PP^1$ such that $E\subset \{o\}\times \PP^1\times \PP^1$ and so $E = \{o\}\times E'$ with $E'\in |\Oo _{\PP^1\times \PP^1}(1,1)|$. Since $Z \subset E\subset \{o\}\times \PP^1\times \PP^1$, there is a $0$-dimensional subscheme $Z'\subset E$ of degree $2$ such that $Z = \{o\}\times Z'$. Since $\deg (A)=2$, $Z$ is the scheme-theoretic intersection of $A$ and $\{o\}\times \PP^1\times \PP^1$. By Theorem \ref{o1} $A$ is smoothable, i.e. there are an integral curve $\Delta$ with $o\in \Delta$ and a flat family $\{A_t\}_{t\in \Delta }$ with $A = A_o$ and $A_t$ a disjoint union of two lines for all $t\in \Delta \setminus \{o\}$. Set $Z_t:= A_t\cap \{o\}\times \PP^1\times \PP^1$. We have $Z_t =\{o\}\times Z'_t$ for a $0$-dimensional subscheme $Z'_t \subset \PP^1\times \PP^1$ of degree $2$ with $Z'_o =Z'$ and $Z'_t$ reduced for all $t\in \Delta \setminus \{o\}$. Fix a general $q\in E'$. Decreasing $\Delta$ if necessary, we may assume $q\notin Z'_t$ for any $t$ and so $|\Oo _{Z'_t\cup \{q\}}(1,1)|$ contains a unique curve, say $E'_t$. Since $q\in E'$, we have $E'_o = E'$. Set $E_t= \{o\}\times E'_t$. The algebraic family $\{A_t\cup E_t\}_{t\in \Delta}$ is a flat family. Since $[A_t\cup E_t]\in \Gamma _1$ for $t\ne o$, so we have $[A_o\cup E_o]\in \overline{\Gamma _1}$.

\quad (b) Now assume that $E$ is a disjoint union of two lines, say $L_1$ of tridegree $(0,1,0)$ and $L_2$ of tridegree $(0,0,1)$. Since $\deg (L_i\cap L)\le 1$ and $L_i$ is a smooth curve, we have $\deg (L_i\cap A)\le 2$. With no loss of generality we may assume $\deg (A\cap L_1)\ge \deg (A\cap L_2)$. We have $\chi (\Oo _C) = \chi (\Oo _A) + 2 -\deg (E\cap A) =a+2-\deg (E\cap A)$ and $a\ge 2$. Thus we get either 
\begin{itemize}
\item $a=2$ with $\deg (L_1\cap A)=2$ and $\deg (L_2\cap A)=1$, or 
\item $a=3$ with $\deg (L_1\cap A)=\deg (L_2\cap A) =2$. 
\end{itemize}

Let us show that the locus $\Gamma_{2b}$ of these types of curves, is contained in $\overline{\Gamma _1}$ (cf. \cite[Proposition 5.10]{CCM}). Note that the space $\Gamma_{2b}$ is a $\PP^1$ (or its open subset)-bundle over $(\PP^1\times \PP^1) \times ((\PP^1\times \PP^1)\setminus\mathrm{D})$, where $\mathrm{D}$ is the diagonal. Here the first $\PP^1\times \PP^1$ parameterizes the supporting line of the double lines $A$ and the second $(\PP^1\times \PP^1)\setminus\mathrm{D}$ parameterizes the ordered pairs $(L_1,L_2)$ of two lines. Also the fiber $\PP^1\cong\PP \Ext^1(\Oo_{C},\Oo_{L}(-1))$ parameterizes the non-split extensions:
\begin{equation}\label{ch1}
0\to \Oo_L(-1) \to \Ff \to \Oo_C \to 0,
\end{equation}
where $C=L\cup L_1\cup L_2$.

Consider the moduli space $\mathbf{M}(X, \beta)$ of stable maps $f:C'\rightarrow X$ of genus $0$ and $f_*[D]=\beta \in H_2(X)$ of tridegree $(2,1,1)$. Let $\Theta_{2b}$ be the locus  of stable maps 
$$f:C'=L'\cup L_1'\cup L_2'\to X$$ 
with $f(C')=C=L\cup L_1\cup L_2$ such that $\mathrm{deg}(f_{|L'})=(2,0,0)$, $\mathrm{deg}(f_{|L_1'})=(0,1,0)$ and $\mathrm{deg}(f|_{L_2'})=(0,0,1)$. Then one can easily see that $\Theta_{2b}$ is a $\PP^2$-bundle over $(\PP^1\times \PP^1) \times ((\PP^1\times \PP^1)\setminus\mathrm{D})$ where $\PP^2$ parameterizes the degree two stable maps on $L$. 
To apply the modification method as in \cite{CK11}, we need to choose a smooth chart of $\mathbf{M}(X, \beta)$ at $[f]$.
In fact, around $[f]$, from \cite[Theorem 0.1]{Parker}, the space of maps $\mathbf{M}(X, \beta)$ can be obtained as the $SL(2)$-quotient
$$\mathbf{M}(X, \beta) \cong \mathbf{M}(\PP^1\times X, (1,\beta))/\Aut (\PP^1)$$
of the moduli space $\mathbf{M}(\PP^1\times X, (1,\beta))$ of stable maps in $\PP^1\times X$ of bidegree $(1,\beta)$ where $\Aut (\PP^1)=SL(2)$ canonically acts on $\mathbf{M}(\PP^1\times X, (1,\beta))$ (cf. \cite[\S 3.1]{CHK}). Among the fiber over $[f]$ along the GIT-quotient map, if we choose a graph map $f'$ such that the restriction on $L'$ is of bidegree $(1,(2,0,0))$ which doubly covers $\PP^1\times L \subset \PP^1\times X$, then $f'$ has the trivial automorphism. Hence, around $[f]$, the space $\mathbf{M}(\PP^1\times X, (1,\beta))$ is a smooth chart which is compatible with the $SL(2)$-action. Thus the argument in \cite[Lemma 4.6]{CK11} about the construction of the Kodaira-Spencer map of the space of maps can be naturally applied in our situation.

Now let us compute the normal space of $\Theta_{2b}$ by the same method as did in the proof of \cite[Lemma 4.10]{CK11}. Consider the long exact sequence:
\begin{align*}
0 &\rightarrow \Ext^0(\Omega_C, \Oo_C)\rightarrow \Ext^0(\Omega_X,\Oo_C)\rightarrow \Ext^0(N_{C/X}^*,\Oo_C)\\
&\stackrel{\psi}{\rightarrow} \Ext^1(\Omega_C, \Oo_C)\rightarrow \Ext^1(\Omega_X,\Oo_C)=0
\end{align*}
The last term is zero by the convexity of $X$ and $ \Ext^0(N_{C/X}^*,\Oo_C)\cong H^0(N_{C/X})\cong \CC^6$ because of the smoothness of $\mathbf{H}(1,1,1,1)$. Since $C$ has two node points, so we get $\Ext^1(\Omega_C, \Oo_C)=\CC^2$. Therefore $\mathrm{ker}(\psi)$ in the above means the deformation of $C$, while keeping the two node points. That is, this is the deformation space of the base space of $\Theta_{2b}$. On the other hand, as a similar computation did in the proof of \cite[Lemma 4.10]{CK11}, we obtain the following commutative diagram
$$
\xymatrix{0\ar[r]&\mathrm{ker}(\psi)\ar[r]\ar[d]^{\zeta}&\Ext^0(N_{C/X}^*,\Oo_C)\ar[r]\ar[d]&\Ext^1(\Omega_C, \Oo_C)\ar[r]\ar[d]^{\cong}&0\\
0\ar[r]&\frac{T_{[f]}\mathbf{M}(X, \beta)}{T_{[f]}\mathbf{M}(C,\beta)}\ar[r]&H^0(f^*N_{C/X}^*,\Oo_{C'})\ar[r]&\Ext^2([f^*\Omega_C\rightarrow \Omega_D],\Oo_{C'})\ar[r]&0.}$$
Hence the normal space of $\Theta_{2b}$ is $\mathrm{coker}(\zeta)$, which is isomorphic to $\Ext^0(N_{C/X}^*,\Oo_{L}(-1))$ obtained from the exact sequence $0\rightarrow \Oo_C\rightarrow{f_*\Oo_{C'}} \rightarrow \Oo_L(-1)\rightarrow 0$. Moreover the Kodaira-Spencer map $T_{[f]}\mathbf{M}(\PP^1\times X, \beta)\rightarrow \Ext_X^1(f_*\Oo_{C'},f_*\Oo_{C'})$ (For the definition, see (4.11) in \cite{CK11}) descends to the normal space which is compatible with the map
$$N_{\Theta_{2b}/\mathbf{M}(X, \beta),[f]}=\Ext^0(N_{C/X}^*,\Oo_{L}(-1))\cong \Ext^0(I_C,\Oo_{L}(-1))\cong \Ext^1(\Oo_C,\Oo_{L}(-1)).
$$
This implies that if we do the modification of the direct image sheaf $f_*\Oo_{C'}$ along the normal direction, the modified sheaf must lie in $\Ext^1(\Oo_C,\Oo_{L}(-1))$ bijectively (cf. \cite[Lemma 4.6]{CK11}). Since $\mathbf{M}(X, \beta)$ is irreducible by \cite{KP01} and $\Gamma_1$ can be regarded as an open subset of $\mathbf{M}(X, \beta)$ due to Lemma \ref{2.1.1.0} and \cite[Theorem 2]{FP97}, so we get that $\Gamma_{2b}\subset \overline{\Gamma_1}$.
\end{proof}

\section{Moduli of pure sheaves of dimension one}

\begin{definition}
Let $\Ff$ be a pure sheaf of dimension $1$ on $X$ with $\chi_{\Ff}(x,y,z)=e_1x+e_2y+e_3z+\chi$. The {\it p-slope} of $\Ff$ is defined to be $p(\Ff)=\chi/(e_1+e_2+e_3)$. $\Ff$ is called {\it semistable (stable)} with respect to the ample line bundle $\Oo_X(1,1,1)$ if, for any proper subsheaf $\Ff '$, we have
$$p(\Ff ')=\frac{\chi'}{e_1'+e_2'+e_3'} \le (<) \frac{\chi}{e_1+e_2+e_3}=p(\Ff)$$
where $\chi_{\Ff '}(x,y,z)=e_1'x+e_2'y+e_3'z+\chi'$. 
\end{definition}

For every semistable $1$-dimensional sheaf $\Ff$ with $\chi_{\Ff}(x,y,z)=e_1x+e_2y+e_3z+\chi$, let us define $C_{\Ff }:=\Supp (\Ff)$ to be its scheme-theoretic support and then it corresponds to $(e_1, e_2, e_3)=e_1t_2t_3+e_2t_1t_3+e_3t_1t_2\in A^2(X)$. We often use slope stability (resp. slope semistability) instead of Gieseker stability (resp. semistability) with respect to $\Ll:=\Oo_X(1,1,1)$, just to simplify the notation; they give the same condition, because the support is $1$-dimensional and so the inequalities for Gieseker and slopes stabilities are the same.

\begin{definition}
Let $\mathbf{M}(e_1, e_2, e_3,\chi)$ be the moduli space of semistable sheaves on $X$ with linear Hilbert tripolynomial $\chi(x,y,z)=e_1x+e_2y+e_3z+\chi$. 
\end{definition}

\begin{remark}
If we let $\mathbf{M}_{X,\PP^7}(\mu, \chi)$ be the moduli space semistable sheaves in Gieseker sense on $\PP^7$ with linear Hilbert polynomial $\chi(t)=\mu t+\chi$, which are $\Oo_X$-sheaves, then we have a natural decomposition 
\begin{equation}
\mathbf{M}_{X,\PP^7}(\mu, \chi)=\bigsqcup_{\substack{0\le e_1, e_2, e_3\le \mu, \\ e_1+e_2+e_3=\mu}} \mathbf{M}(e_1, e_2, e_3,\chi).
\end{equation}
\end{remark}

\begin{remark}
Since $\chi_{\Ff(a,b,c)}(x,y,z)=\chi_{\Ff}(x,y,z)+(e_1a+e_2b+e_3c)$, so we have an isomorphism 
$$\mathbf{M}(e_1, e_2, e_3,\chi) \to \mathbf{M}(e_1, e_2, e_3,\chi+e_1a+e_2b+e_3c)$$
defined by $\Ff \mapsto \Ff(a,b,c)$. Thus we may assume that $0<\chi\le \gcd (e_1, e_2, e_3)$.
\end{remark}

For a positive rational number $\alpha \in \QQ_{>0}$, a pair $(s,\Ff)$ of a non-zero section $s:\Oo_X \rightarrow \Ff$ of a sheaf $\Ff$ is called $\alpha$-{\it semistable} if $\Ff$ is pure and for any non-zero proper subsheaf $\Ff' \subset \Ff$, we have
$$\frac{\chi(\Ff'(t))+\delta \cdot \alpha}{r(\Ff')} \leq \frac{\chi (\Ff (t))+\alpha}{r(\Ff)}$$
for $t\gg0$. Here $r(\Ff)$ is the leading coefficient of the Hilbert polynomial $\chi(\Ff(t))$ and we take $\delta=1$ if the section $s$ factors through $\Ff'$ and $\delta=0$ if not. As usual, if the inequality is strict, we call it $\alpha$-{\it stable}. By \cite[Theorem 4.2]{He} the wall happens at $\alpha$ with which the strictly $\alpha$-semistability occurs. As a routine, we will write $(1,\Ff)$ for the pair of a sheaf with $\Ff$ with a non-zero section and $(0,\Ff)$ for the pair of sheaf with zero section . 

Note that there are only finitely many critical values $\{\alpha_1, \ldots, \alpha_s \}$ for $\alpha$-stability with $\alpha_1<\cdots<\alpha_s$. Then any $\alpha\in (\alpha_i, \alpha_{i+1})$ gives the same moduli spaces of $\alpha$-stable pairs. Notice that if $\alpha<\alpha_1$, then $\alpha$-stability is equivalent to the Gieseker stability and so there exists a forgetful surjection to the moduli of stable sheaves. If $\alpha>\alpha_s$, then the cokernel of the pair $\Oo_X \rightarrow \Ff$ is supported at a $0$-dimensional subscheme and so we get the moduli of PT stable pairs. 

Let us denote the moduli of $\alpha$-stable pairs with $\alpha<\alpha_1$ by $\mathbf{M}^{0+}(e_1, e_2, e_3, \chi)$ and with $\alpha>\alpha_s$ by $\mathbf{M}^{\infty}(e_1, e_2, e_3, \chi)$. Then there exists a forgetful surjection $\mathbf{M}^{0+}(e_1, e_2, e_3, \chi)\rightarrow \mathbf{M}(e_1, e_2, e_3, \chi)$ and $\mathbf{M}^{\infty}(e_1, e_2, e_3, \chi)$ is the moduli space of PT stable pairs.

\begin{lemma}\label{lemm}
For each $[C]\in \mathbf{H}(1,1,1,1)$, its structure sheaf $\Oo_C$ is stable. 
\end{lemma}
\begin{proof}
The assertion is clear if $C$ is irreducible. Assume that $C$ is not irreducible. Let $\Ff$ be an unstabilizing subsheaf of $\Oo_C$ and so we have $h^0(\Ff)>0$. A non-trivial section of $H^0(\Ff)$ induces an injection $\Oo_D \rightarrow \Ff$ with the cokernel $\Gg$, where $D$ is a subcurve of $C_{\Ff}$ and so of $C$. In particular we have $\deg (D)=1$ or $2$. Now we get a non-trivial map $\Oo_{C_{\Gg}} \rightarrow \Gg$ from the natural surjection $\Oo_C \rightarrow \Oo_{C_{\Gg}}$. It is impossible, because $h^1(\Oo_D)=0$ and $h^0(\Oo_C)=1$.
\end{proof}

\begin{proposition}\label{a3}
We have $$\mathbf{M}^{0+}(1,1,1,1)\cong \mathbf{M}^\alpha(1,1,1,1) \cong \mathbf{M}^{\infty}(1,1,1,1)$$
for all $\alpha>0$. 
\end{proposition}
\begin{proof}
It is enough to show that there is no wall-crossing among moduli spaces. By \cite[Theorem 4.2]{He} the wall occurs at $\alpha$ with which the strict $\alpha$-semistability occurs. Assume that $(s,\Ff')$ is a subsheaf of $(1,\Ff)$ which induces the strict $\alpha$-semistability and whose support has minimal degree. Hence we may assume that $(s,\Ff ')$ is $\alpha$-stable. Since $(1,\Ff )$ is $\alpha$-semistable, the quotient $ (s',\Gg ):= (1,\Ff )/(s,\Ff ')$ is $\alpha$-semistable.

\quad {(a)} If the Hilbert polynomial of $\Ff'$ is $t+c$ for a constant $c$, then we have $\Ff' \cong \Oo_L(c-1)$ for a line $L$ with $c+\delta \cdot \alpha=\frac{1+\alpha}{3}$. If $\delta=1$, then we have $3c+2\alpha=1$. But since $\Ff'$ has a non-zero section, we must have $c\ge 1$, a contradiction. If $\delta=0$, then we have $\alpha=3c-1>0$ and so $c\ge 1$. Its cokernel is $(1,\Gg)$ where $\Gg$ is a sheaf with the Hilbert polynomial $2t+1-c$. Note that $\chi(\Gg)=1-c\le 0$ and $\Gg$ has a non-zero section. If the schematic support $C_{\Gg}$ of $\Gg$ is a smooth conic, then we have $\Gg \cong \Oo_{C_{\Gg}}(-c)$ and so it cannot have a non-zero section. 

\quad (a1) Now assume that $C_{\Gg}$ is a singular conic, say $A_1 \cup A_2$ with $\{o\}:= A_1\cap A_2$. First assume that $\Gg$ is locally free and set $a_i:= \deg (\Gg_{|A_i})$ with $a_1\ge a_2$. Since $\Gg$ is locally free, we have $\deg (\Gg) =a_1+a_2$ and so $a_1+a_2 =-c \le -1$. Since $h^0(\Gg)>0$, so we have $a_1\ge 1$ and $(1, \Oo_{A_1}(a_1))$ gives a contradiction. Now assume that $\Gg$ is not locally free. Since it has pure rank $1$, so its torsion $\tau$ is supported at $\{o\}$ with $\deg (\tau ) =1$. The integers $b_i :=\deg (\Gg _{|A_i}/\mathrm{Tors}(\Gg _{|A_i}))$ satisfy $b_1+b_2-1 =-c\le -1$ and so $b_1+b_2\le 0$. Without loss of generality we may assume $b_1\ge b_2$. Since $h^0(\Gg )>0$ and $\Gg$ is not locally free, we have $b_1>0$ (the case $(b_1,b_2)=(0,0)$ has no section). Also note that if the section of $\Gg$ is zero at a general point of $A_1$, then we would have $b_2>0$, which is not possible. Let $\Nn$ be the kernel of the restriction map $\Gg \to \Gg _{|A_2}$. Since $b_2 <0$, the non-zero section of $\Gg$ induces a non-zero section of $\Nn$. Note that $\Nn$ is an $\Oo _{A_1}$-sheaf. Since $\Gg$ has depth $1$, so is $\Nn$. Thus $\Nn$ is a rank $1$ locally free $\Oo _{A_1}$-sheaf. Since $\deg (\tau )=1$, we have $\Nn \cong \Oo _{A_1}(b_1)$. Therefore $(1,\Gg)$ has a subsheaf $(1,\Oo _{A_1}(b_1))$ and it contradicts the semistability of $(1,\Gg )$.

\quad (a2) So the only possibility of $\Gg$ is $\Gg \cong \Oo_{B_1}(k)\oplus \Oo_{B_2}(-k-c)$ with $k\ge 0$ and two skew lines $B_1,B_2$. But then $\Gg$ has $(1,\Oo_{B_1}(k))$ as a subsheaf, contradicting the semistability with slope $t+c$ of $(1,\Gg )$, because $t+k+3c> t+c$.

\quad{(b)} Now we assume that the Hilbert polynomial of $\Ff'$ is $2t+c$. By assumption $(s,\Ff ')$ is $\alpha$-stable. We have $\frac{c+\delta \cdot \alpha}{2}=\frac{1+\alpha}{3}$. If $\delta=0$, then we have $3c=2+2\alpha\ge 3$ and so $c\ge 1$. It implies that the cokernel pair is $(1,\Gg)$ with $t+1-c$ as its Hilbert polynomial. Since $\Gg \cong \Oo_L(-c)$ with $L$ a line, it cannot have a non-zero section, a contradiction. Assume that $\delta =1$ and then we have $3c+\alpha=2$. Thus we have $c=0$ and $\alpha=2$. Lemma \ref{a1} gives that $C:= C_{\Ff '}$ is reduced. With no loss of generality we may assume that $C$ has tridegree $(1,1,0)$. Since $\deg ({C})$ is the leading term of the Hilbert polynomial of $\Ff'$, we get that $\Ff'$ has rank $1$ at a general point of each irreducible component of $C$. First assume that $C$ is smooth. Since $\chi(\Ff')=0$ and $\Ff'$ has a non-zero section, as in (a) the only possibility is $\Ff'\cong \Oo_{L_1}(k)\oplus \Oo_{L_2}(-k-2)$ with $k\ge 0$ and $L_1,L_2$ two skew lines. Then $(1,\Oo_{L_1}(k))$ contradicts the semistability of $(1,\Ff)$. Now assume that $C$ is a reduced conic, say $C=A_1\cup A_2$ with $A_1$ of tridegree $(1,0,0)$.  Set $b_i :=\deg (\Ff' _{|A_i}/\mathrm{Tors}(\Ff' _{|A_i}))$, $i=1,2$. Notice that we have an inclusion $\Oo _{A_i}(b_i-1)\hookrightarrow \Ff'$. With no loss of generality we may assume $b_1\ge b_2$. First assume that $\Ff$ is locally free. In this case $\Ff '_{|A_i}$ has no torsion and $b_1+b_2 =c-1 =-1$. We get $b_2\le -1$. Since $h^0(\Ff ')>0$ we get $b_1>0$ and that $s$ induces a non-zero section $s'$ of $\Oo _{A_1}(b_1-1)$. The $2$-slope of $(s',\Oo _{A_1}(b_1-1))$ is at least $2$, contradicting the $2$-stability of $(s,\Ff ')$. Now assume that $\Ff '$ is not locally free. Since $\Ff '$ has pure rank $1$, we have $b_1+b_2=0$. Since $h^0(\Ff ') >0$ and $\Ff'$ is not locally free, we have $(b_1,b_2) \ne (0,0)$. Hence $b_2<0$ and $b_1>0$. We conclude as in the locally free case.
\end{proof}

\begin{corollary}
We have $$\mathbf{H}(1,1,1,1) \cong \mathbf{M}^{\alpha}(1,1,1,1) \cong \mathbf{M}(1,1,1,1)$$
for all $\alpha>0$. In particular they are all irreducible, smooth and unirational varieties of dimension $6$. 
\end{corollary}
\begin{proof}
Fix any $[C]\in \mathbf{H}(1,1,1,1)$. We saw in the proof of Proposition \ref{a2} that $C$ is reduced and connected and so $h^0(\Oo _C)=1$. Take $(1,\Oo _C)$. When $C$ is irreducible, then obviously it gives an element of  $\mathbf{M}^{\alpha}(1,1,1,1)$ for all $\alpha >0$. Hence we see that
a non-empty open part of $\mathbf{H}(1,1,1,1)$ survives in $\mathbf{M}^{0+}(1,1,1,1)$ and in particular the latter is non-empty. For $[\Ff] \in \mathbf{M}(1,1,1,1)$, there exists a non-zero section inducing an injection $0 \rightarrow \Oo_{C_{\Ff}} \rightarrow \Ff$ due to the isomorphism $\mathbf{M}^{0+}(1,1,1,1) \cong \mathbf{M}^{\infty}(1,1,1,1)$ by Proposition \ref{a3}. We know that $[C_{\Ff}] \in \mathbf{H}(1,1,1,1)$. Thus we have $\chi_{\Oo_{C_{\Ff}}}(x,y,z)=x+y+z+1$ and so $\Ff \cong \Oo_{C_{\Ff}}$. Since $h^0(\Oo _C)=1$ for all $[C]\in \mathbf{H}(1,1,1,1)$, so we have $\mathbf{M}^{0+}(1,1,1,1) \cong \mathbf{M}(1,1,1,1)$. Since $\mathbf{H}(1,1,1,1)$ and $\mathbf{M}^{0+}(1,1,1,1)$ are projective and a non-empty open subset of the irreducible variety $\mathbf{H}(1,1,1,1)$ survives in $\mathbf{M}^{0+}(1,1,1,1)$, we get a birational surjection $\mathbf{M}(1,1,1,1)\rightarrow \mathbf{H}(1,1,1,1)$. Indeed this map is an isomorphism due to Lemma \ref{lemm}, which gives its inverse map. Then the assertion follows from Proposition \ref{a2}.
\end{proof}

\section{Segre threefold with Picard number two}

In this section we take $X:= \PP^2\times \PP^1$ and in most cases we adopt the same notations as in the case of $X =\PP^1\times \PP^1\times \PP^1$. For a locally CM curve $C\subset X$ with pure dimension $1$, the {\it bidegree} $(e_1,e_2)\in \ZZ^{\oplus 2}$ is defined to be the pair $(e_1, e_2)$ of integers $e_1:= \deg (\Oo _C(1,0))$ and $e_2:= \deg (\Oo _C(0,1))$, where degree is computed using the Hilbert function of the $\Oo _X$-sheaves $\Oo _C(1,0)$ and $\Oo _C(0,1)$ with respect to the ample line bundle $\Oo _X(1,1)$. We also say that $C=\emptyset$ has bidegree $(0,0)$. Since $\Oo _C(1,0)$ and $\Oo _C(0,1)$ are spanned, we have $e_1\ge 0$ and $e_2\ge 0$, i.e. $(e_1,e_2)\in \NN ^{\oplus 2}$. We have $\deg ({C}) = \deg (\Oo _C(1,1)) =e_1+e_2$. Note that the bipolynomial of $\Oo_C$ is of the form $e_1x+e_2y+\chi$ for some $\chi \in \ZZ$ and $\chi = \chi (\Oo _C)$.

As in the proof of Lemma \ref{a1} we get the following.
\begin{lemma}\label{a1.1}
Let $C\subset X$ be a locally CM curve with the bidegree $(e_1, e_2)$. If the bidegree of $C_{\mathrm{red}}$ is $(b_1, b_2)$ with $b_i=0$ for some $i$, then we have $e_i=0$.
\end{lemma}

\begin{proposition}\label{c1}
$\mathbf{H}(1,1,1)$ is smooth and irreducible of dimension $5$, and its elements are reduced. 
\end{proposition}

\begin{proof}
Let us fix $[C]\in \mathbf{H}(1,1,1)$. By Lemma \ref{a1.1} every $1$-dimensional component of $C$ is generically reduced, i.e. the purely $1$-dimensional subscheme $E$ of $C_{\mathrm{red}}$ has bidegree $(1,1)$. We have $\chi (\Oo _C) \ge \chi (\Oo _D)$ for each connected component $D$ of $E$ and equality holds if and only if $D =C$. Since we have $\chi (\Oo _D) \ge 1$, we get $C = D$ and that $C$ is connected. If $C$ is irreducible, then it is a smooth conic. Since $N_C$ is a quotient of $TX_{|C}$, we get $h^1(N_C) =0$. It implies that $\mathbf{H}(1,1,1)$ is smooth at $[C]$ and of dimension $h^0(N_C) = \deg (N_C)+2 = \deg (TX_{|C}) = 5$. Indeed we have $N_C \cong \Oo_{\PP^1}(2)\oplus \Oo_{\PP^1}(1)$.

Now assume that $C$ is reducible, say the union of a line $D_1$ of bidegree $(1,0)$ and a line $D_2$ of bidegree $(0,1)$. Since $\deg (D_1\cap D_2) \le 1$ and $[C]\in \mathbf{H}(1,1,1)_+$, we get $\deg (D_1\cap D_2)=1$ and that $C$ is nodal. Since $h^1(TX_{|C})=0$ and the natural map $TX_{|C} \rightarrow N_C$ is supported at the point $D_1\cap D_2$, we have $H^1(N_C)=0$. Hence $\mathbf{H}(1,1,1)$ is again smooth at $[C]$ and of dimension $\deg (N_C) +2=5$. Since the set of all such reducible curves is $4$-dimensional, so each such curve is in the closure of the open subset of $\mathbf{H}(1,1,1)$ parametrizing the smooth curves.
\end{proof}

\begin{proposition}\label{c2}
$\mathbf{H}(1,1,2)$ is smooth and irreducible of dimension $5$. It parametrizes the disjoint unions of two lines, one of bidegree $(1,0)$ and the other of bidegree $(0,1)$.
\end{proposition}

\begin{proof}
By Lemma \ref{a1.1} any curve $[C]\in \mathbf{H}(1,1,2)_+$ is reduced. If $C$ is irreducible, we get $\chi (\Oo _C)=1$, a contradiction. If $C=D_1\cup D_2$ with lines $D_1$ of bidegree $(1,0)$ and $D_2$ of bidegree $(0,1)$, we get $D_1\cap D_2 =\emptyset$. Then we have $h^1(N_C)=0$ and $h^0(N_C) = h^0(N_{D_1})+h^0(N_{D_2}) = \deg(TX_{|D_1})+\deg (TX_{|D_2}) =5$. 
\end{proof}

\begin{remark}
Using the argument in the proof of Proposition \ref{c1}, we get that $\mathbf{H}(1,1,\chi)_+ =\emptyset$ if either $\chi\le 0$ or $\chi\ge 3$.
\end{remark}

In the case of Segre threefold with Picard number three, the main ingredient is the knowledge on the Hilbert scheme of double lines. So we suggest the following results for the Segre threefold with Picard number two, as in Theorem \ref{o1}. As in the case of $\PP^1 \times \PP^1 \times \PP^1$, let $\Dd _a$ be the subset of $\mathbf{H}(0,2,a)_+$ parametrizing the double lines whose reduction is a line of bidegree $(0,1)$ in $X=\PP^2 \times \PP^1$ for each $a\in \ZZ$. For the moment we take $\Dd _a$ as a set and it would be clear in each case which scheme-structure is used on it. Since $X$ is a smooth 3-fold, \cite[Remark 1.3]{bf} says that each $[B]\in \Dd _a$ is obtained by the Ferrand's construction and in particular it is a ribbon in the sense of \cite{be} with a line of bidegree $(1,0)$ as its support. Let $\Rr _a$ be the subset of $\mathbf{H}(2,0,a)_+$ parametrizing the double structures on lines of bidegree $(1,0)$. 

\begin{proposition}\label{uo1.1}
The description on $\Rr_a$ is as follows:
\begin{enumerate}
\item $\Rr_a$ is non-empty if and only if $a\ge 2$. It is parametrized by an irreducible and rational variety of dimension $2a-1$. 
\item We have $\Rr_a = \mathbf{H}(2,0,a)_+$ for $a\ge 3$.
\item $\mathbf{H}(2,0,2)_+$ is smooth, irreducible, rational and of dimension $4$. 
\end{enumerate}
\end{proposition}
 
\begin{proof}
Each element of $\Rr _a$ is a ribbon in the sense of \cite{be}. For any line $L\subset X$ of bidegree $(1,0)$ let $\Rr _a(L)$ denote the set of all $[A]\in \Rr _a$ such that $A_{\mathrm{red}} =L$. The set of all lines of $X$ with bidegree $(1,0)$ is isomorphic to $\PP^1$. Any line $L\subset X$ of bidegree $(1,0)$ has trivial normal bundle and so $\Rr _a(L)$ is parametrized by the pairs $(f,g)$ with $f\in H^0(\Oo _L(a-2))$ and $g\in H^0(\Oo _L(a-2))$ with no common zero. Here we have the convention that $(L,f_1,g_1)$ and $(L,f_2,g_2)$ give the same element of $\Rr _a(L)$ if and only if there is $t\in \CC^{\times}$ with $f_1=tg_1$ and $f_2=tg_2$. Hence we get parts (1) and (2) of Proposition \ref{uo1.1} for $a\ge 3$ and that for any $a\ge 2$ each element of $\Rr _a$ is a split ribbon.
 
The set $\mathbf{H}(2,0,2)_+$ is the disjoint union of $\Rr _2$ and the set $\Tt$ of all disjoint unions of two different lines of bidegree $(1,0)$. $\Tt$ is isomorphic to the symmetric product of two copies of $\PP^2$ and so it is smooth and rational with $\dim (\Tt )=4$. Fix a line $L\subset X$ of bidegree $(1,0)$ such that $L = \PP^1\times \{o\}$ with $o\in \PP^2$, and $[A]\in \Rr _2(L)$ determined by $(f,g)\in \CC^2\setminus \{(0,0)\}$, up to a non-zero scalar. The pair $(f,g)$ defines a degree $2$ zero-dimensional scheme $v\subset \PP^2$ with $v_{\mathrm{red}} = \{o\}$. Let $R\subset \PP^2$ be the line spanned by $v$ and then $A$ is contained in $L\times R$ as a curve of bidegree $(2,0)$ and hence it is a flat deformation of a family of pairs of disjoint lines of $L\times R$ and hence of $X$. We also get that the normal sheaf $N_A$ of $A$ in $X$ is isomorphic to $\Oo _L\oplus \Oo _L(1)$. Hence we get $h^1(N_A) =0$ and so $\mathbf{H}(2,0,2)_+$ is smooth at $[A]$.
 \end{proof}

Below we give a description on $\Dd_a$ as in Theorem \ref{o1} and Proposition \ref{uo1.1}, which can be proven by the same way. 

\begin{proposition}\label{o1.1}
The description on $\Dd_a$ is as follows:
\begin{enumerate}
\item $\Dd_a$ is non-empty if and only if $a\ge 2$. It is parametrized by an irreducible variety of dimension $2a-1$. 
\item We have $\Dd _a = \mathbf{H}(0,2,a)_+$ for $a\ge 3$.
\item $\mathbf{H}(0,2,2)_+$ is smooth, irreducible, rational and of dimension $4$. 
\end{enumerate}
\end{proposition}

\def\cprime{$'$}
\providecommand{\bysame}{\leavevmode\hbox to3em{\hrulefill}\thinspace}
\providecommand{\MR}{\relax\ifhmode\unskip\space\fi MR }
\providecommand{\MRhref}[2]{%
  \href{http://www.ams.org/mathscinet-getitem?mr=#1}{#2}
}
\providecommand{\href}[2]{#2}

\end{document}